\documentclass[oneside]{amsart}
\usepackage[T1]{fontenc}
\usepackage[latin9]{inputenc}
\usepackage{refstyle}
\usepackage{amstext}
\usepackage{amsthm}
\usepackage{amssymb}
\usepackage{graphicx,color}

\makeatletter

\AtBeginDocument{\providecommand\eqref[1]{\ref{eq:#1}}}
\RS@ifundefined{subsecref}
  {\newref{subsec}{name = \RSsectxt}}
  {}
\RS@ifundefined{thmref}
  {\def\RSthmtxt{theorem~}\newref{thm}{name = \RSthmtxt}}
  {}
\RS@ifundefined{lemref}
  {\def\RSlemtxt{lemma~}\newref{lem}{name = \RSlemtxt}}
  {}

\numberwithin{equation}{section}
\numberwithin{figure}{section}
\theoremstyle{plain}
\newtheorem{thm}{\protect\theoremname}
\theoremstyle{plain}
\newtheorem{prop}[thm]{\protect\propositionname}
\theoremstyle{plain}
\newtheorem{lem}[thm]{\protect\lemmaname}
\theoremstyle{plain}
\newtheorem{cor}[thm]{\protect\corollaryname}
\theoremstyle{remark}

\theoremstyle{plain}

\usepackage{tikz}
\usetikzlibrary{arrows,positioning,shapes}

\makeatother

\providecommand{\claimname}{Claim}
\providecommand{\corollaryname}{Corollary}
\providecommand{\lemmaname}{Lemma}
\providecommand{\propositionname}{Proposition}
\providecommand{\theoremname}{Theorem}
\DeclareMathOperator{\trace}{trace}
\DeclareMathOperator{\Ker}{Ker}
\DeclareMathOperator{\Range}{Range}
\DeclareMathOperator{\Span}{Span}

\begin{document}
\global\long\def\dom{B_{1}\setminus B_{\epsilon}}%
\global\long\def\thetf{\theta}%
\title{On optimal cloaking-by-mapping transformations}
\author{Yves Capdeboscq}
\address{Université de Paris and Sorbonne Université, CNRS, Laboratoire Jacques-Louis Lions (LJLL),
F-75006 Paris, France}
\email{yves.capdeboscq@u-paris.fr}
\author{Michael S. Vogelius}
\address{Department of Mathematics, Rutgers University, New Brunswick,  New Jersey 08901,  USA} 
\email{vogelius@math.rutgers.edu}
%\date{\empty}

\maketitle

\begin{abstract}
A central ingredient of cloaking-by-mapping is the diffeomorphisn which transforms an annulus with a small hole into an annulus with a finite size hole, while being the identity on the outer boundary of the annulus. The resulting meta-material is anisotropic, which makes it difficult to manufacture. The problem of minimizing anisotropy among radial transformations has been studied in \cite{gries-vogel}. In this work, as in \cite{gries-vogel}, we formulate the problem of minimizing anisotropy as an energy minimization problem. Our main goal is to provide strong evidence for the conjecture that for cloaks with circular boundaries, non-radial transformations do not lead to lower degree of anisotropy. In the final section, we consider cloaks with non-circular boundaries and show that in this case, non-radial cloaks may be advantageous, when it comes to minimizing anisotropy. 
\end{abstract}

\section{Introduction}

\noindent
A central ingredient in the construction of (approximate) cloaks by the passive cloaking technique, known as ``cloaking by mapping'', is the diffeomorphism, which transforms an annulus
with a small hole into an annulus with a finite size hole,
and which is the identity on the outer boundary of the annulus.
The push-forward of the background coefficient (say, the identity
matrix) with the diffeomorphism represents the meta-material
needed for the cloak, and the finite size hole is the area that may
be used as a ``hiding place'' \cite{ksvw}. The fact that the diffeomorphism is the identity on the outer boundary ensures that the perturbation in the ``far field'' is that corresponding to a small inhomogeneity. The corresponding ``lack of cloaking''/visibility can be estimated by the volume of the small inhomogeneity. The required meta-material
is anisotropic, which presents a problem when it comes to actual
manufacture of the cloak. Typically a radial affine transformation has been used \cite{greenlasuhl2, greenlasuhl, ksvw, MILT-CLOAK, pen-schu-smi},
however, a very natural question arises, namely : ``are there transformations
that lead to lower degree of
anisotropy than the radial affine transformation?\,'' In \cite{gries-vogel} it was shown that there are indeed better radial transformations than the affine, when it comes to minimizing anisotropy. In that paper the meta-material obtained by ``optimal radial transformation'' is also shown to be quite related to meta-materials obtained by other cloak enhancement strategies, employing additional layers \cite{ammari-kang, heu-vogel}.
The focus of this note is to produce very strong evidence for the
conjecture that when the cloak takes the shape of a classical annulus, non-radial transformations do not help in reducing
the degree of anisotropy. Like in \cite{gries-vogel}, we formulate the problem of minimizing
anisotropy as a variational problem (minimization of an appropriate
energy). Corollary \ref{cor:7} summarizes our main results. Broadly speaking, we show that 
\begin{itemize}
\item There exists a radial transformation, which is a stationary point for the energy.
\item This radial transformation has smaller energy than all other transformations
with ``directional field'' $\frac{x}{|x|}$.
\item If the amplitude is kept fixed and radial, then any change in the
``directional field'' away from $\frac{x}{|x|}$ will increase energy. 
\end{itemize}
In the final section of this note we consider the case when the outer (and inner) boundary of the cloak are not circles, and we illustrate how the optimal radial transformation for the circular case translates into a non-radial (optimal) transformation for a non-circular cloak.

\section{Preliminaries}

\noindent
For $r>0$ we set 
\[
B_{r}=\left\{ x\in\mathbb{R}^{2}\,:\,\left|x\right|<r\right\} ~,~\text{ and }~C_{r}=\left\{ x\in\mathbb{R}^{2}\,:\,\left|x\right|=r\right\} ~.
\]
Given $\epsilon>0$, we shall use the notation $\Phi$ for a bijective diffeomorphism $\overline{B_{1}}\setminus B_{\epsilon} \rightarrow \overline{B_{1}}\setminus B_{\frac{1}{2}} $ with  $\Phi\in C^{1}\left(\overline{B_{1}}\setminus B_{\epsilon};\overline{B_{1}}\setminus B_{\frac{1}{2}}\right)~,$ and $\Phi^{-1}\in C^{1}\left(\overline{B_{1}}\setminus B_{\frac{1}{2}};\overline{B_{1}}\setminus B_{\epsilon}\right)$.
We furthermore impose that 
\[
\left.\Phi\right|_{C_{1}}=Id~~,\text{ and }\Phi\left(C_{\epsilon}\right)=C_{\frac{1}{2}}~.
\]
One such transformation is the radial affine transformation, given by
\[
x\rightarrow\left(\frac{|x|-1}{2(1-\epsilon)}+1\right)\frac{x}{|x|}~.
\]
The push-forward of the identity matrix with the diffeomorphism $\Phi$
is given by
\[
\Phi_{\text{*}}\left[I\right](\Phi(x))=\frac{D\Phi D\Phi^{T}}{\left|\det D\Phi\right|}(x)~.
\]
This is a positive definite matrix, and since we are in two dimensions,
with determinant $1$. Let $0<\lambda_{1}(x)\le1\le \lambda_{2}(x)$ denote
the eigenvalues of $\Phi_{\text{*}}\left[I\right](\Phi(x))$. A natural
measure of the degree of anisotropy of $\Phi_{\text{*}}\left[I\right]$ at the point
$\Phi(x)$ is 
\begin{eqnarray*}
\left|\lambda_{1}(x)-1\right|+\left|\lambda_{2}(x)-1\right|=\lambda_{2}(x)-\lambda_{1}(x) & = & \sqrt{\left(\lambda_{2}(x)-\lambda_{1}(x)\right)^{2}}\\
 & = & \sqrt{\left(\lambda_{1}(x)+\lambda_{2}(x)\right)^{2}-4}~.
\end{eqnarray*}
To minimize this we must minimize $\trace\Phi_{\text{*}}\left[I\right](\Phi(x))$.
As a way of minimizing the aggregate anisotropy we shall seek to minimize\footnote{In a slight deviation from \cite{gries-vogel}, the domain of integration of the energy functional is $B_1\setminus B_\epsilon$, not the transformed domain $B_1\setminus B_{\frac12}$.}
\[
I_p(\Phi)=\int_{B_{1}\setminus B_{\epsilon}}\left(\trace\Phi_{\text{*}}\left[I\right]\right)^{p}(\Phi(x))~dx
\]
for a fixed choice of $1\le p<\infty$, and 
\[
I_\infty(\Phi)=\max_{x\in \overline{B_{1}}\setminus B_{\epsilon}}\trace\Phi_{\text{*}}\left[I\right](\Phi(x))= \max_{y\in \overline{B_{1}}\setminus B_{\frac12}}\trace\Phi_{\text{*}}\left[I \right](y)   ~,
\]
corresponding to $p=\infty$. Note that $\lambda$ is an eigenvalue for $\Phi_{\text{*}}\left[I \right](\Phi(x))$,
with eigenvector $v$, if and only if $\lambda$ is an eigenvalue
for 
\[
\frac{D\Phi^{T}D\Phi}{\left|\det D\Phi\right|}(x)~,
\]
with eigenvector $D\Phi^{T}(x)v$, and thus 
\[
\trace\Phi_{\text{*}}\left[I \right](\Phi(x))=\trace\left[\frac{D\Phi^{T}D\Phi}{\left|\det D\Phi\right|}\right](x)~.
\]

\begin{prop}
Let $\Phi$ be represented in terms of its polar decomposition 
\[
\Phi=\exp(\psi)\phi~,
\]
where the directional field $\phi$ is in $C^{1}(\overline{B_{1}}\setminus B_{\epsilon};\mathcal{S}^{1})$
and logarithmic amplitude $\psi$ is in $C^{1}(\overline{B_{1}}\setminus B_{\epsilon};\mathbb{R})$.
Then 
\[
\trace\left(D\Phi^{T}D\Phi\right)=\left|\Phi\right|^{2}\left(\left|D\phi\right|^{2}+\left|D\psi\right|^{2}\right)~.
\]
\end{prop}

\begin{proof}
Differentiating we find 
\[
D\Phi=\exp(\psi)\phi D\psi^{T}+\exp(\psi)D\phi~.
\]
Since $\phi^{T}\phi=1,$ we have 
\[
\phi^{T}D\phi=0~,\text{ and }D\phi^{T}\phi=0~,
\]
and therefore 
\begin{align*}
D\Phi^{T}D\Phi & =\exp(2\psi)\left(D\psi\phi^{T}+D\phi^{T}\right)\left(\phi D\psi^{T}+D\phi\right)\\
 & =\left|\Phi\right|^{2}\left(D\phi^{T}D\phi+D\psi D\psi^{T}\right)~.
\end{align*}
By taking the trace we arrive at the desired conclusion. 
\end{proof}
It is well known that $\phi$, being in $C^{1}\left(\overline{B_{1}}\setminus B_{\epsilon};\mathcal{S}^{1}\right)$,
admits a canonical lift $\theta= \arg \left( \phi \right)\in C^{1}\left(\overline{B_{1}}\setminus B_{\epsilon};\mathbb{R}/2\pi\mathbb{Z}\right)$ \footnote{A function $\theta\,:\,\overline{B_{1}}\setminus B_{\epsilon} \rightarrow \mathbb{R}/2\pi\mathbb{Z}$ is an element of  $C^{1}\left(\overline{B_{1}}\setminus B_{\epsilon};\mathbb{R}/2\pi\mathbb{Z}\right)$ iff given any point $x\in \overline{B_{1}}\setminus B_{\epsilon}$ there exists an open neighborhood $\omega_x$ of $x$, relative to $\overline{B_{1}}\setminus B_{\epsilon}$, and a representative of $\theta$ (mod $2\pi$) that lies in $C^{1}(\omega_x;\mathbb{R})$. Notice that the globally defined derivative of $\theta \in C^1\left(\overline{B_{1}}\setminus B_{\epsilon};\mathbb{R}/2\pi\mathbb{Z}\right)$, $D\theta$, lies in $C^0\left(\overline{B_{1}}\setminus B_{\epsilon};\mathbb{R}^2\right)$.}
such that 
\[
\phi=\left(\cos\theta,\sin\theta\right)^{T}.
\]
We write 
\[
J=\left[\begin{array}{cc}
0 & -1\\
1 & 0
\end{array}\right],\quad\text{e}_{r}=\frac{x}{\left|x\right|}~,~~\text{ and }~~\text{e}_{\theta}=J\frac{x}{\left|x\right|}.
\]

\begin{prop}
The matrix $D\phi$ has rank one; furthermore 
\[
\Range(D\phi)=\Span(\phi)^{\perp}~,~\text{ and }~\Ker(D\phi)=\Span(D\theta)^{\perp}~.
\]
We denote by $\widehat{D\psi,D\theta}$ the angle defined by 
\begin{align*}
\cos\left(\widehat{D\psi,D\theta}\right) & =\frac{1}{\left|D\psi\right|\left|D\theta\right|}D\psi\cdot D\theta~,\text{ and }\\
\sin\left(\widehat{D\psi,D\theta}\right) & =\frac{1}{\left|D\psi\right|\left|D\theta\right|}\det\left(D\psi,D\theta\right)\cdot
\end{align*}
Then 
\[
\trace \Phi_{\text{*}}\left[I\right](\Phi(x))=\frac{1}{\left|\sin\left(\widehat{D\psi,D\theta}\right)\right|}\left(\frac{\left|D\theta\right|}{\left|D\psi\right|}+\frac{\left|D\psi\right|}{\left|D\theta\right|}\right)(x)\geq\left(\frac{\left|D\theta\right|}{\left|D\psi\right|}+\frac{\left|D\psi\right|}{\left|D\theta\right|}\right)(x)
\]
with equality only when $D\psi\cdot D\theta=0$. 
\end{prop}

\begin{proof}
We calculate 
\[
D\phi=\left(J\phi\right)D\theta^{T},
\]
which immediately leads to the statements about $\text{Range}(D\phi)$
and $\Ker(D\phi)$, and which also gives 
\[
D\theta=\left(D\phi\right)^{T}\left(J\phi\right)~.
\]
As a consequence 
\begin{align*}
\det D\Phi & =\det\left(\phi D\psi^{T}+\left(J\phi\right)D\theta^{T}\right)\exp\left(2\psi\right)\\
 & =\det\left(D\psi,D\theta\right)\left|\Phi\right|^{2}\\
 & =\left|D\theta\right|\left|D\psi\right|\sin\left(\widehat{D\psi,D\theta}\right)\left|\Phi\right|^{2}.
\end{align*}
Here we have used that $\det D\Phi\ne0$, since $\Phi$ is a bijective
diffeomorphism of $\overline{B_{1}}\setminus B_{\epsilon}$ onto $\overline{B_{1}}\setminus B_{\frac{1}{2}}$;
consequently $\det\left(D\psi,D\theta\right)\ne0$ and $\left|D\psi\right|\left|D\theta\right|>0$
and $\sin\left(\widehat{D\psi,D\theta}\right)$ (and $\widehat{D\psi,D\theta}$)
is well-defined. It now follows that 
\begin{align*}
\trace\Phi_{\text{*}}\left[I \right](\Phi(x)) & =\frac{\left(\left|D\theta\right|^{2}+\left|D\psi\right|^{2}\right)}{\left|D\theta\right|\left|D\psi\right|\left|\sin\left(\widehat{D\psi,D\theta}\right)\right|}(x)\\
 & =\frac{1}{\left|\sin\left(\widehat{D\psi,D\theta}\right)\right|}\left(\frac{\left|D\theta\right|}{\left|D\psi\right|}+\frac{\left|D\psi\right|}{\left|D\theta\right|}\right)(x)\geq\frac{\left|D\theta\right|}{\left|D\psi\right|}(x)+\frac{\left|D\psi\right|}{\left|D\theta\right|}(x)~,
\end{align*}
with equality if and only of $D\psi$ is normal to $D\theta$, and
therefore in the kernel of $D\phi$. 
\end{proof}

\section{The radial transformation case}

\noindent
For the general case of a radial transformation $\phi=\frac{x}{\left|x\right|}$, and $\psi=f\left(\left|x\right|\right).$
Then $D\theta=\frac{1}{\left|x\right|}J\frac{x}{\left|x\right|}$
and $D\psi=f^{\prime}\left(\left|x\right|\right)\frac{x}{\left|x\right|}.$
The transformation 
\[
\Phi=\exp(\psi)\phi
\]
is a bijective $C^{1}$ diffeomorphism of $\overline{B_{1}}\setminus B_{\epsilon}$
onto $\overline{B_{1}}\setminus B_{\frac{1}{2}}$ with
\[
\left.\Phi\right|_{C_{1}}=Id~~,\text{ and }\Phi\left(C_{\epsilon}\right)=C_{\frac{1}{2}}~,
\]
if and only if 
\[
f(\epsilon)=-\log2~,f(1)=0~,\text{ and }f\in C^{1}([\epsilon,1])\text{ with }f'(r)>0\text{ for all }r\in[\epsilon,1].
\]
In this case, $\sin\left(\widehat{D\psi,D\theta}\right)=1$, and 
\[
\trace\Phi_{\text{*}}\left[I \right](\Phi(x))=\frac{1}{\left|x\right|f^{\prime}\left(\left|x\right|\right)}+\left|x\right|f^{\prime}\left(\left|x\right|\right)~.
\]

\begin{prop}
\label{1dE-L}
Suppose $1\leq p<\infty$, and let $I_{p}$ denote the energy 
\[
I_{p}\left(f\right):=\int_{B_{1}\setminus B_{\epsilon}}\left(\trace\Phi_{\text{*}}\left[I \right]\right)^{p}(\Phi(x))\text{d}x=2\pi\int_{\epsilon}^{1}\left(\frac{1}{rf^{\prime}\left(r\right)}+rf^{\prime}\left(r\right)\right)^{p}rd\text{r}~,
\]
with values in $(0,\infty)$, defined on the convex set 
\[
\mathcal{C}=\biggl\{ f\in C^{1}\left([\epsilon,1]\right)\text{ : }f^{\prime}>0~~,\,f\left(\epsilon\right)=-\log2,f\left(1\right)=0\biggr\}~.
\]
Then 
\begin{itemize}
\item $I_{p}$ has a unique minimizer, $f_{p}$, in $\mathcal{C}$. 
\item $f_{p}$ lies in $C^{\infty}([\epsilon,1])$, and is the unique solution in $\mathcal{C}$
to the Euler--Lagrange equation 
\begin{align*}
\left(\left(\frac{1}{rf_{p}^{\prime}\left(r\right)}+rf_{p}^{\prime}\left(r\right)\right)^{p-1}\left(-\frac{1}{\left(f_{p}^{\prime}\right)^{2}}+r^{2}\right)\right)^{\prime} & =0\text{ in }\left[\epsilon,1\right]~.\hspace{20pt}\text{(E-L)}
\end{align*}
\end{itemize}
\end{prop}

\begin{proof}
We start by establishing (part of) the last statement concerning the existence of a unique solution to the Euler--Lagrange equation (E-L). By integration, any $C^1$ solution to (E-L) must satisfy 
\[
G(rf_{p}'(r))=\frac{C}{r^{2}}
\]
for some constant $C$, with the function $G:\mathbb{R}_{+}\rightarrow\mathbb{R}$
given by 
\[
G(t)=\left(\frac{1}{t}+t\right)^{p-1}\left(-\frac{1}{t^{2}}+1\right)~.
\]
Now suppose $1<p<\infty$. A simple calculation shows that $G$ is monotonically increasing,
with $G(1)=0$, $\lim_{t\rightarrow0_{+}}G(t)=-\infty$ and $\lim_{t\rightarrow\infty}G(t)=\infty$.
$G^{-1}:~\mathbb{R}\rightarrow\mathbb{R}_{+}$ is thus well defined,
and $f_{p}$ has the form 
\[
f_{p}(r)=\int_{\epsilon}^{r}f_{p}'(t)dt-\log2=\int_{\epsilon}^{r}t^{-1}G^{-1}\left(\frac{C}{t^{2}}\right)~dt-\log2~,
\]
for some constant $C$. The constant $C$ must be chosen so that $f_{p}$
satisfies the boundary condition $f_{p}(1)=0$. As $C\rightarrow\int_{\epsilon}^{1}t^{-1}G^{-1}\left(\frac{C}{t^{2}}\right)~dt-\log2$ is continuous and monotonically increasing, with 
\[
\int_{\epsilon}^{1}t^{-1}G^{-1}\left(\frac{C}{t^{2}}\right)~dt-\log2\rightarrow\begin{cases}
|\log\epsilon|-\log2>0 & \text{ when }C\rightarrow 0\\
-\log2<0 & \text{ when }C\rightarrow-\infty
\end{cases}~,
\]
it follows immediately that there exists a unique value $C_{0}<0$ for which the boundary condition $f_{p}(1)=0$
is satisfied. This shows the uniqueness of the solution to the Euler--Lagrange equation in $\mathcal{C}$.
Furthermore, the formula 
\[
f_{p}(r)=\int_{\epsilon}^{r}t^{-1}G^{-1}\left(\frac{C_{0}}{t^{2}}\right)~dt-\log2
\]
clearly gives rise to a $C^\infty$ function in $\mathcal{C}$ which solves the equation (E-L), thus establishing the existence. A slightly modified argument works for $p=1$, and in that case we find the (even more) explicit formula
$$
f_{1}:r\to\log\left(\frac{3r+\sqrt{9r^{2}+16\left(2-\epsilon\right)\left(\frac{1}{2}-\epsilon\right)}}{4\left(2-\epsilon\right)}\right)~.
$$

\noindent
We now proceed to show that $f_p$ is the unique minimizer of $I_p$ in $\mathcal{C}$. Since the function $(0,\infty)\ni x\rightarrow\left(\frac{1}{x}+x\right)^{p}\in(0,\infty)$ is strictly convex, it follows immediately that $I_{p}$ is strictly convex on $\mathcal{C}$. Now suppose there existed a function $g \in \mathcal{C}$ with $I_p(g)<I_p(f_p)$. The convexity of the functional $I_p$ implies that
$$
\frac{d}{d\tau}|_{\tau=0} I_p(f_p+\tau(g-f_p))\le I_p(g)-I_p(f_p)<0~,
$$
or
$$
\int_{\epsilon}^{1}\left(\frac{1}{rf^{\prime}\left(r\right)}+rf^{\prime}\left(r\right)\right)^{p-1}\left(-\frac{1}{\left(f_{p}^{\prime}\right)^{2}}+r^{2}\right)(g-f_p)^{\prime}d\text{r} <0~,
$$
in contradiction with the fact that $f_p$ is a solution to the Euler-Lagrange equation (E-L). This verifies that $f_p$ is a minimizer of $I_p$ in $\mathcal{C}$. The fact that the minimizer is unique follows immediately from the strict convexity of $I_p$.
\end{proof}

\noindent
\begin{proof}[\bf Remark]

\noindent
The logarithmic amplitude $f_{1}$ gives rise to the transformation 
\[
\Phi_{1}=\left(\frac{3\left|x\right|+\sqrt{9\left|x\right|^{2}+16\left(2-\epsilon\right)\left(\frac{1}{2}-\epsilon\right)}}{4\left(2-\epsilon\right)}\right)\text{\ensuremath{\frac{x}{\left|x\right|}}}~.
\]
We compute 
\[
I_{1}(f_{1})=2\pi\int_{\epsilon}^{1}\left(\frac{1}{f_{1}^{\prime}\left(r\right)}+r^{2}f_{1}^{\prime}\left(r\right)\right)dr=2\pi\left(1-\epsilon^{2}+\frac{2}{3}\left(2\epsilon-1\right)^{2}\right)~.
\]
By comparison, the radial affine transformation 
\[
\Phi_{ra}=\left(\frac{\left|x\right|-1}{2(1-\epsilon)}+1\right)\text{\ensuremath{\frac{x}{\left|x\right|}}},
\]
with logarithmic amplitude 
\[
f_{ra}(r)=\log\left(\frac{r-1}{2(1-\epsilon)}+1\right)~.
\]
has
\[
I_{1}\left(f_{ra}\right)=2\pi\int_{\epsilon}^{1}\left(\frac{1}{f_{ra}^{\prime}\left(r\right)}+r^{2}f_{ra}^{\prime}\left(r\right)\right)dr=2\pi\left(1-\epsilon^{2}+\ln2\left(2\epsilon-1\right)^{2}\right)\geq I_{1}(f_1)~.
\]
Equality occurs only when $\epsilon=\frac{1}{2}$ (when the associated
transformations are both the identity).
\end{proof}
\vskip 15pt

\noindent
Turning to maximum norm, we consider the minimization 
\[
\mathcal{I}_{\infty}=\inf_{f\in \mathcal{C}}\sup_{\left[\epsilon,1\right]}\left(\frac{1}{rf^{\prime}\left(r\right)}+rf^{\prime}\left(r\right)\right)~.
\]
We note that 
\begin{align*}
\mathcal{I}_{\infty} & =\inf_{K>1}\left\{ \frac{1}{K}+K~:~\exists f\in \mathcal{C}\text{ with }\sup_{r\in\left[\epsilon,1\right]}\left\{ \frac{1}{rf^{\prime}\left(r\right)}+rf^{\prime}\left(r\right)\right\} \leq\frac{1}{K}+K~\right\} \\
 & \ge\inf_{K>1}\left\{ \frac{1}{K}+K~:~\exists f\in \mathcal{C}\text{ with }\frac{1}{K}\left|\log r\right|\leq\left|f\left(r\right)\right|\right\} \\
 & \ge\inf \left\{ \frac{1}{K}+K~:~\frac{\left|\log\epsilon\right|}{\log2}\leq K\right\} =\frac{\log2}{\left|\log\epsilon\right|}+\frac{\left|\log\epsilon\right|}{\log2}~.
\end{align*}
Here we have used that, if $f\in \mathcal{C}$ and if $K>1$, then
\begin{eqnarray*}
\frac{1}{rf^{\prime}\left(r\right)}+rf^{\prime}\left(r\right)\le\frac{1}{K}+K \text{ in }(\epsilon,1) &\implies &\frac{1}{Kr}\le f^{\prime}\left(r\right)\le\frac{K}{r} \text{ in }(\epsilon,1)\\
 &\implies &\frac{1}{K}\left|\log r\right|\le\left|f(r)\right|\le K\left|\log r\right| \text{ in }(\epsilon,1)~.
\end{eqnarray*}
On the other hand, the function 
\begin{equation}
f_{\infty}(r)=\frac{\log2}{|\log\epsilon|}\log r\label{eq:def-finty}
\end{equation}
lies in $\mathcal{C}$, and has $I_{\infty}(f_{\infty})=\frac{\log2}{|\log\epsilon|}+\frac{|\log\epsilon|}{\log2}$.
It now follows immediately that $f_{\infty}$ is a minimizer of $I_{\infty}$
in $\mathcal{C}$. The following graph shows the logarithmic amplitudes  $f_{ra}$ (dashed
orange line), $f_{1}$, $f_{2}$, $f_{3}$, $f_{5}$, $f_{8},f_{13}$
and $f_{\infty}$ (solid lines from red to blue), for $\epsilon=1/100$.
\begin{center}
\includegraphics[width=0.5\columnwidth]{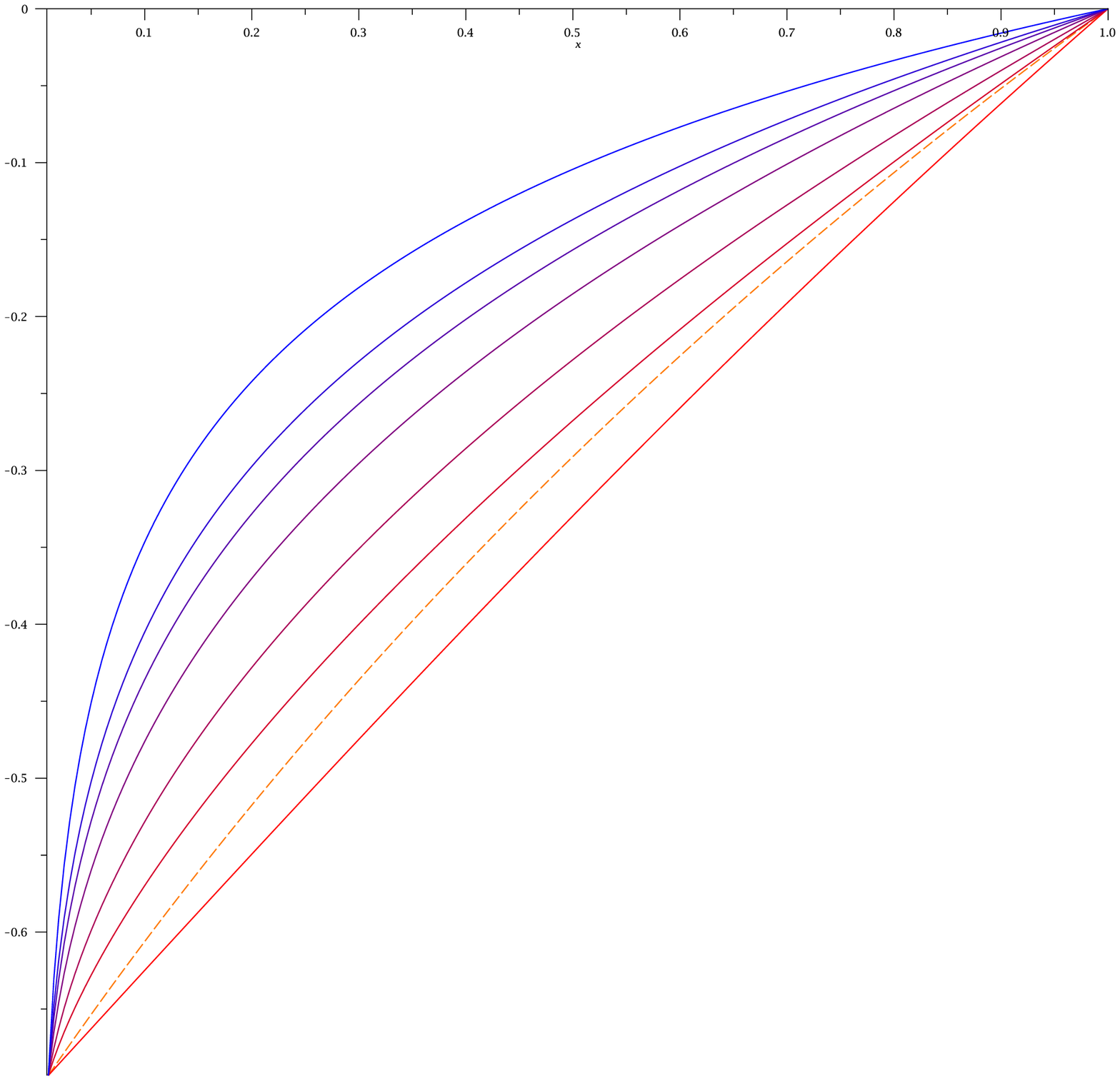}
\par\end{center}

\section{Optimality of radial transforms}

We now return to the general, two dimensional case. By introducing
$u=\psi$ and $V=-JD\theta$ in the formula 
\[
\trace\Phi_{\text{*}}\left[I \right](\Phi(x))=\frac{\left|D\psi\right|^{2}+\left|D\theta\right|^{2}}{\det\left(D\psi,D\theta\right)}(x)~,
\]
we obtain 
\[
\trace\Phi_{\text{*}}\left[I \right](\Phi(x))=\frac{\left|Du\right|^{2}+\left|V\right|^{2}}{Du\cdot V}(x)~.
\]
Similarly, by introducing $u=\theta$ and $V=JD\psi$, we obtain 
\[
\trace\Phi_{\text{*}}\left[I \right](\Phi(x))=\frac{\left|Du\right|^{2}+\left|V\right|^{2}}{Du\cdot V}(x)~.
\]
We thus notice that the problem of minimizing 
\[
I_{p}(\Phi)=\int_{\dom}\left(\trace\Phi_{\text{*}}\left[I \right]\right)^{p}(\Phi(x))\text{d}x
\]
with respect to $\psi$ given $\theta$, and with respect to $\theta$,
given $\psi$ merely differs by a change of the convex test set for $u$ (essentially relating to boundary conditions). Let $\arg\in C^{\infty}(\overline{B_{1}}\setminus B_{\epsilon};\mathbb{R}/2\pi\mathbb{Z})$\footnote{The space $C^{\infty}(\overline{B_{1}}\setminus B_{\epsilon};\mathbb{R}/2\pi\mathbb{Z})$ is defined as 
$\{\, u \in C^{1}(\overline{B_{1}}\setminus B_{\epsilon};\mathbb{R}/2\pi\mathbb{Z})~:~ Du \in C^{\infty}(\overline{B_{1}}\setminus B_{\epsilon};\mathbb{R}^2)\,\}$. Similarly $C^{2,\alpha}(\overline{B_{1}}\setminus B_{\epsilon};\mathbb{R}/2\pi\mathbb{Z})=\{\, u \in C^{1}(\overline{B_{1}}\setminus B_{\epsilon};\mathbb{R}/2\pi\mathbb{Z})~:~ Du \in C^{1,\alpha}(\overline{B_{1}}\setminus B_{\epsilon};\mathbb{R}^2)\,\}$} denote the standard argument function. We introduce the convex sets
\begin{eqnarray*}
\mathcal{C}_{\theta}&=&C^{2,\alpha}(\overline{B_1}\setminus B_\epsilon;\mathbb{R}/2\pi\mathbb{Z})\cap\{~u|_{C_{1}}=\arg~\}\text{ and }\\
\mathcal{C}_{\psi}&=&C^{2,\alpha}(\overline{B_1}\setminus B_\epsilon;\mathbb{R})\cap\{~u|_{C_{\epsilon}}=-\log 2~,~u|_{C_{1}}=0~\}~,
\end{eqnarray*}
for some fixed $\alpha>0$.

\begin{prop}
\label{prop:ConvexFS}Given $\mathcal{C}=\mathcal{C}_{\psi}$ and a fixed $V\in C^0\left(\overline{B_1}\setminus B_\epsilon;\mathbb{R}^2\right)$,
or $\mathcal{C}=\mathcal{C}_{\theta}$ and a fixed $V\in C^0\left(\overline{B_1}\setminus B_\epsilon;\mathbb{R}^2\right)$,
and given $n\geq1$, we introduce 
\[
\mathcal{C}_{n}=\left\{ u\in \mathcal{C}:Du\cdot V\geq\frac{1}{n}\text{ and }\left\Vert u\right\Vert _{C^{2,\alpha}\left(\overline{B_1}\setminus B_\epsilon\right)}\leq n\right\} .
\]
Suppose $\mathcal{C}_{N_0}\neq\emptyset,$ for some $N_0\geq 1$. Given any $1\le p<\infty$, the functional $F_p:\mathcal{C}_{n}\to\mathbb{R}$, $n\ge N_0$,  defined by 
\[
u\to F_p(u)=\int_{B_{1}\setminus B_{\epsilon}}\left(\frac{\left|Du\right|^{2}+\left|V\right|^{2}}{Du\cdot V}\right)^{p}\text{d}x
\]
is strictly convex, continuous, and attains its infimum on $\mathcal{C}_n$ at a unique minimizer. If the unique
minimizer, $u$, lies in $\text{int}(\mathcal{C}_{n})$\footnote{The interior is formed relative to  $\mathcal{C}_\psi$ or $\mathcal{C}_\theta$ with the $C^{2,\alpha}$ topology, respectively. }, then it
satisfies the  associated Euler-Lagrange equation
\begin{equation}
\text{div}\left(\left(\frac{\left|Du\right|^{2}+\left|V\right|^{2}}{Du\cdot V}\right)^{p-1}\left(\frac{2Du}{Du \cdot V}-\frac{\left|Du\right|^{2}+\left|V\right|^{2}}{\left(Du\cdot V\right)^{2}}V\right)\right)=0\text{ in }\dom~,\label{eq:Euler-Lagrange-FP}
\end{equation}
and in the case $\mathcal{C}= \mathcal{C}_{\theta}$, the additional  boundary condition
\begin{equation}
\label{eq:Euler-Lagrange-BC}
\left(\frac{2Du}{Du\cdot V}-\frac{\left|Du\right|^{2}+\left|V\right|^{2}}{\left(Du\cdot V\right)^{2}}V\right)\cdot \frac{x}{|x|}=0 \text{ on } C_\epsilon~.
\end{equation}
Conversely, if there exists a solution to \eqref{Euler-Lagrange-FP} (and \eqref{Euler-Lagrange-BC} in case $\mathcal{C}=\mathcal{C}_\theta$) which lies in $\mathcal{C} \cap \{ Du \cdot V >0 \text{ on } \overline{B_1}\setminus B_\epsilon ~\}$,
then, for some $N\ge 1$, this is the unique minimizer of $F_p$ in $\mathcal{C}_n$, for any $n \ge N$. Consequently this $u$ is also the unique minimizer of $F_p$ in $\mathcal{C} \cap \{ Du \cdot V >0 \text{ on } \overline{B_1}\setminus B_\epsilon ~\}$.
\end{prop}
 
For the proof of Proposition \ref{prop:ConvexFS} we shall need the following lemma. 
\begin{lem}
\label{lem:convexCp}For any $1\le p <\infty$, and any $A>0$, the function
$G_{p}\left[A\right]:\left(0,\infty\right)\times\mathbb{R}\to\mathbb{R}_{+}$,
given by 
\[
\left(x,y\right)\to\left(\frac{A}{x}+\frac{x}{A}+\frac{x}{A}\left(\frac{y}{x}\right)^{2}\right)^{p}
\]
is convex. Furthermore, 
\[
G_{p}\left[A\right](x,y)-\frac{2A^{4}}{\left(A^{2}+M^{2}\right)^{3}}\left(x^{2}+y^{2}\right)
\]
is convex on $B_{M}=\{(x,y)~:~x^{2}+y^{2}<M^{2}~\}$. 
\end{lem}

\begin{proof}
The function $x \to \frac{A}{x}+\frac{x}{A}$ is strictly convex and positive
valued on $\left(0,\infty\right)\times\mathbb{R}.$ The map $\left(x,y\right)\to\frac{1}{A}\frac{y^2}{x}$
is convex and positive on $\left(0,\infty\right)\times\mathbb{R}.$
Indeed, its Hessian has eigenvalues $0$ and $\frac{2}{A}\frac{x^{2}+y^{2}}{x^{3}}.$
The sum of two convex (and positive valued) functions is convex (and
positive valued), and the composition of it with $z\to z^{p}$, a monotonically
increasing and convex function on $\left(0,\infty\right)$, results
in a convex (positive valued) function.

To establish the second assertion, we compute lower bounds for $D^{2}G_{p}\left[A\right]$.
It is a fact that the lowest eigenvalue of a symmetric positive definite
matrix is bounded below by the quotient of the determinant over the
trace. We compute that for $p\geq1$, 
\[
\frac{\det\left(D^{2}G_{p}\left[A\right]\right)}{\text{tr}\left(D^{2}C_{p}\left[A\right]\right)}>\frac{4p}{p+1}G_{p}\left[A\right]\frac{A^{4}}{\left(A^{2}+x^{2}+y^{2}\right)^{3}}\geq4\frac{A^{4}}{\left(A^{2}+x^{2}+y^{2}\right)^{3}}.
\]
In particular, on the ball $B_{M}=\{(x,y)~:~x^{2}+y^{2}<M^{2}~\}$
we have 
\[
D^{2}G_{p}\left[A\right](x,y)>\frac{4A^{4}}{\left(A^{2}+M^{2}\right)^{3}}I~.
\]
This immediately leads to the second assertion of the lemma. 
\end{proof}
We are now ready for the proof of Proposition \ref{prop:ConvexFS}.
\begin{proof} Given $u\in \mathcal{C}_{n}$,
we define 
\[
P_{V}\left(Du\right)=Du\cdot\frac{V}{\left|V\right|},\text{ and }P_{V^{\perp}}\left(Du\right)=Du\cdot\frac{JV}{\left|V\right|}~.
\]
Then 
\begin{eqnarray*}
\left(\frac{\left|Du\right|^{2}+\left|V\right|^{2}}{Du\cdot V}\right)^p&=&\left(\frac{\left|V\right|}{P_{V}\left(Du\right)}+\frac{P_{V}\left(Du\right)}{\left|V\right|}+\frac{P_{V}\left(Du\right)}{\left|V\right|}\left(\frac{P_{V^{\perp}}\left(Du\right)}{P_{V}\left(Du\right)}\right)^{2}\right)^p \\
&=&G_p[\,|V|\,]\left(P_V(Du),P_{V^\perp}(Du)\right)~.
\end{eqnarray*}
Note that $\mathcal{C}_{N_0}\neq\emptyset$ implies $\inf\left|V\right|>0$. On
$\mathcal{C}_{n}$, $\left|P_{V}\left(Du\right)\right|^{2}+\left|P_{V^{\perp}}\left(Du\right)\right|^{2}\leq n^{2}$, and therefore
for any $u,v \in \mathcal{C}_{n}$, $n\ge N_0$, and any $\tau \in [0,1]$
\begin{align*}
 & G_{p}\left[\,|V|\,\right]\left(P_{V}\left(D\left(\tau u+\left(1-\tau\right)v\right)\right),P_{V^{\perp}}\left(D \left(\tau u+\left(1-\tau\right)v\right)\right)\right)\\
\leq\tau & G_{p}\left[\,|V|\,\right]\left(P_{V}\left(Du\right),P_{V^{\perp}}\left(Du\right)\right)+\left(1-\tau\right)G_{p}\left[\, |V| \,\right]\left(P_{V}\left(Dv\right),P_{V^{\perp}}\left(Dv\right)\right)\\
- & \tau\left(1-\tau\right)K\left|D\left(u-v\right)\right|^{2}~,
\end{align*}
with
\[
K=\frac{2\inf\left|V\right|^{4}}{\left(n^{2}+\sup\left|V\right|^{2}\right)^{3}}>0~.
\]
For $u,v\in \mathcal{C}_{n}$, and $\tau \in\left[0,1\right]$, we thus get
\[
F_p\left(\tau u+\left(1-\tau\right)v\right)\leq\tau F_p\left(u\right)+\left(1-\tau\right)F_p\left(v\right)-\tau\left(1-\tau\right)K\int_{\dom}\left|D\left(u-v\right)\right|^{2}~dx~,
\]
and so $F_p$ is strictly convex on $\mathcal{C}_{n}$. In regards to continuity, let $u_m$ be a sequence in $\mathcal{C}_n$ with $u_{m}\to u$ in the $C^1$ topology. Then the functions
\[
x\to G_{p}\left[\,|V|\,\right]\left(P_{V}\left(Du_{m}\right),P_{V^{\perp}}\left(Du_{m}\right)\right)(x)
\]
are measurable, non negative, uniformly bounded, and converge pointwise to the function
\[
x\to G_{p}\left[\,|V|\,\right]\left(P_{V}\left(Du\right),P_{V^{\perp}}\left(Du\right)\right)(x)~.
\]
Thanks to the Lebesgue Dominated Convergence Theorem, this implies 
\[
\lim F_p\left(u_{m}\right)=F_p\left(u\right)~.
\]
Since $\mathcal{C}_{n}$ is compact with respect to the $C^1$ topology, the $C^1$ continuity of $F_p$ implies the existence of a minimizer. The convexity of $\mathcal{C}_{n}$  and the strict convexity of $F_p$ yields the uniqueness of the minimizer. A computation
shows that for any $u\in \mathcal{C}_{n}$, $F_p$ is Gâteaux-differentiable at $u$,
and its differential is given by 
\begin{align*}
 & \left\langle DF_p(u),h\right\rangle \\
 & =\int_{\dom}p\left(\frac{\left|Du\right|^{2}+\left|V\right|^{2}}{Du\cdot V}\right)^{p-1}\left(\frac{2Du}{Du\cdot V}-\frac{\left|Du\right|^{2}+\left|V\right|^{2}}{\left(Du\cdot V\right)^{2}}V\right)\cdot Dh \, \text{d}x~,
\end{align*}
for $h\in C^{1}$.
Note that $u\in \mathcal{C}_{n}$ is the unique minimizer if and only if for
all $v\in \mathcal{C}_{n}$ there holds 
\begin{equation}
\left\langle DF_p(u),v-u\right\rangle \geq0.\label{eq:minimum}
\end{equation}
If the minimizer lies in the interior of $\mathcal{C}_{n}$, \eqref{minimum}
implies 
\[
\left\langle DF_p(u),h\right\rangle =0
\]
for all $h\in C^{2,\alpha}\cap \{ h=0 \text{ on } C_\epsilon \text{ and } C_1~\}$, if $\mathcal{C}= \mathcal{C}_\psi$, and for all $h\in C^{2,\alpha}\cap \{ h=0 \text{ on } C_1~\}$, if $\mathcal{C}= \mathcal{C}_\theta$\, ; in other words, $u$ satisfies the Euler-Lagrange
equation \eqref{Euler-Lagrange-FP} (or \eqref{Euler-Lagrange-FP} and  \eqref{Euler-Lagrange-BC} when $\mathcal{C}= \mathcal{C}_\theta$). Conversely, if $w\in \mathcal{C} \cap \{ Du \cdot V >0 \hbox{ on } \overline{B_1}\setminus B_\epsilon \,\}$
satisfies \eqref{Euler-Lagrange-FP} (and \eqref{Euler-Lagrange-BC} if $\mathcal{C}=\mathcal{C}_\theta$), then, for some $N$, it lies in $\mathcal{C}_n$ for all $n\ge N$, and it satisfies $\left\langle DF_p(w),v-w\right\rangle =0$ (in particular $\ge 0$) for all $v \in \mathcal{C}_n$; $w$ is thus the unique minimizer of $F_p$ in $\mathcal{C}_n$ for any $n\ge N$. It follows immediately that $w$ is a minimizer of $F_p$ in  $\mathcal{C} \cap \{ Du \cdot V >0 \hbox{ on } \overline{B_1}\setminus B_\epsilon \,\}$. The uniqueness of this minimizer follows from the strict convexity of $F_p$ on $\mathcal{C}_n$ for any $n$.
\end{proof}
\begin{cor}
\label{cor:6}
A global $C^{2,\alpha}$ minimizer $(\psi,\theta)$ of $I_p$, subject to $\psi = -\log 2$ at $|x|=\epsilon$, $\psi= 0$ and $\theta=\arg$ at $|x|=1$, and $\det\left(D\psi,D\theta\right) >0$
on $\overline{B_{1}}\setminus B_{\epsilon}$, satisfies 
\[
\text{div}\left(\left(\frac{\left|D\psi\right|^{2}+\left|D\theta\right|^{2}}{\det\left(D\psi,D\theta\right)}\right)^{p}\left(\frac{2D\psi}{\left|D\psi\right|^{2}+\left|D\theta\right|^{2}}+\frac{JD\theta}{\det\left(D\psi,D\theta\right)}\right)\right)=0,
\]
and 
\[
\text{div}\left(\left(\frac{\left|D\psi\right|^{2}+\left|D\theta\right|^{2}}{\det\left(D\psi,D\theta\right)}\right)^{p}\left(\frac{2D\theta}{\left|D\psi\right|^{2}+\left|D\theta\right|^{2}}-\frac{JD\psi}{\det\left(D\psi,D\theta\right)}\right)\right)=0.
\]
Furthermore, 
\[
\left(\frac{\left|D\psi\right|^{2}+\left|D\theta\right|^{2}}{\det\left(D\psi,D\theta\right)}JD\psi-2D\theta\right)\cdot\frac{x}{\left|x\right|}=0\text{ on }\left\{ \left|x\right|=\epsilon\right\} .
\]
\end{cor}

\begin{proof}
The $\psi$ component of this global minimizer automatically lies in $\text{int}(\mathcal{C}_n)$ with $\mathcal{C}=\mathcal{C}_\psi$ and $V= -J D\theta$ for some $n$, and it is a minimizer of $F_p$ in $\mathcal{C}_n$. The first equation of this corollary is now simply the Euler-Lagrange \eqref{Euler-Lagrange-FP} for such a minimizer. Similarly, the $\theta$ component of this global minimizer  lies in $\text{int}(\mathcal{C}_n)$ with $\mathcal{C}=\mathcal{C}_\theta$ and $V= J D\psi$ for some $n$, and is a minimizer of $F_p$ in $\mathcal{C}_n$. The two last equations of this corollary are simply the Euler-Lagrange \eqref{Euler-Lagrange-FP} and the boundary condition (\ref{eq:Euler-Lagrange-BC}) satisfied by such a minimizer.  
\end{proof}

\begin{cor}
\label{cor:7}Let $f_p$ be the function introduced in Proposition \ref{1dE-L}. The transformation $x \to f_p(|x|)\frac{x}{|x|}$, or rather the function pair $(f_p(|x|), \arg(x))$ satisfies the three Euler-Lagrange equations from Corollary \ref{cor:6}. As a consequence
\begin{equation}
\label{firstineq}
I_{p}(f_p(|x|)\frac{x}{|x|}) \le I_{p}(\psi(x)\frac{x}{|x|})~, 
\end{equation}
for any $\psi \in \mathcal{C}_\psi\cap \{ D\psi(x)\cdot \frac{x}{|x|}>0 \hbox{ on } \overline{B_1}\setminus B_\epsilon \,\}$~. The last two Euler-Lagrange equations from Corollary \ref{cor:6} are actually satisfied by any pair $(f(|x|),\arg(x))$, with $f \in \{f \in C^{2,\alpha}([\epsilon,1])~:~ f'>0~, f(\epsilon)=-\log 2~, f(1)=0~\}$. As a consequence
\begin{equation}
\label{secondineq}
I_p(f_p(|x|)\frac{x}{|x|})\le I_p(f(|x|)\frac{x}{|x|})\le I_p(f(|x|)\phi(x))~,
\end{equation}
for any $\phi(x)=(\cos(\theta(x)),\sin (\theta(x))^t$, with $\theta\in \mathcal C_\theta \cap \{ D\theta \cdot J\frac{x}{|x|}>0 \hbox{ on } \overline{B_1}\setminus B_\epsilon \,\}$ and any $f \in \{f \in C^{2,\alpha}([\epsilon,1])~:~ f'>0~, f(\epsilon)=-\log 2~, f(1)=0~\}$.
\end{cor}
\begin{proof}
Direct calculations verify that the first Euler-Lagrange equation from Corollary \ref{cor:6} is satisfied by $(f_p(|x|), \arg(x))$, and that the last two Euler-Lagrange equations from Corollary \ref{cor:6} are satisfied by any pair $(f(|x|),\arg(x))$, with $f \in \{f \in C^{2,\alpha}([\epsilon,1])~:~ f'>0~, f(\epsilon)=-\log 2~, f(1)=0~\}$. The inequality (\ref{firstineq}) now follows immediately from the last statement in Proposition \ref{prop:ConvexFS} in the case $\mathcal{C}= \mathcal{C}_\psi$ and $V=-JD\arg(x)= \frac{1}{|x|}\frac{x}{|x|}$. The first inequality in (\ref{secondineq}) is a direct consequence of (\ref{firstineq}). The second inequality follows from the last statement in Proposition \ref{prop:ConvexFS} in the case $\mathcal{C}= \mathcal{C}_\theta$ and $V= JDf(|x|)= f'(|x|)J\frac{x}{|x|}$.
\end{proof}

\begin{figure}
     \begin{tikzpicture}
        [
        squarednode/.style={%
            ellipse,  %rectangle,
            draw=black!60,
            fill=white,
            very thick,
            minimum size=5mm,
            text centered,
            text width=3cm,
           % node distance=2cm
        }
        ]
        %Nodes
        \node[squarednode]      (radial)                              {\[F_p\left( e^{\psi(x)}\frac{x}{\left|x\right|} \right)\]};
        \node[squarednode]      (general)       [above=of radial] {\[F_p\left( e^{\psi(x)}\varphi(x)               \right)\]};
        \node[squarednode]      (minimal)       [right=of radial] {\[F_p\left(e^{f_p\left(\left|x\right|\right)}\frac{x}{\left|x\right|}\right)\]};
        \node[squarednode]      (radial2)       [above=of minimal] {\[F_p\left(e^{f\left(\left|x\right|\right)}\varphi(x)\right)\]};

        %Lines
        \draw[thick,->] (radial.east) -- node [above,midway] {$\geq$} (minimal.west);
        \draw[thick,->] (radial2.south) -- node [above,midway,sloped] {$\geq$}(minimal.north);
        \end{tikzpicture}
\caption{Illustration of the conclusions of Corollary~\ref{cor:7}. }

\end{figure}
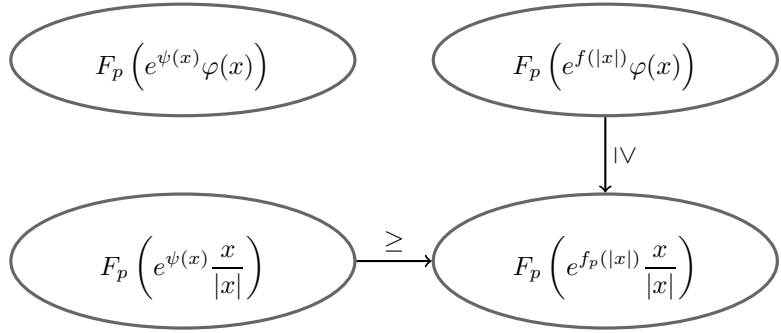

\section{Optimal cloaks for simply connected domains}

So far our study has focused on the situation where the cloaks are constructed from diffeomorphisms of the classical annulus $\overline{B_1}\setminus B_\epsilon$ to the classical annulus $\overline{B_1}\setminus B_{\frac{1}{2}}$, and the corresponding push-forwards of the identity matrix. In a more general setting, one could consider instead three simply
connected domains, $\omega_{\epsilon}\subset\omega_{\frac{1}{2}}\subset\Omega$
containing the origin (where $\omega_{\epsilon}$ is comparable to
$B_{\epsilon}$) and a bijective diffeomorphism $\Psi_{\epsilon}: \overline{\Omega}\setminus \omega_{\epsilon} \to \overline{\Omega}\setminus \omega_{\frac12}$,
such that $\Psi_{\epsilon}=Id$ on $\partial\Omega$ and $\Psi_{\epsilon}\left(\partial \omega_{\epsilon}\right)=\partial \omega_{\frac{1}{2}}$. As before, the material parameters of the cloak would be the push-forward of $I$ by $\Psi_\epsilon$. Any smooth globally minimizing transformation would still satisfy the Euler-Lagrange equations of Corrollary \ref{cor:6}, if we continue to use the energy $I_p$.

\vskip 10pt
\noindent
The goal of this section is to show that for general geometries one should (naturally) not expect the optimal transformations to be radial. As we demonstrate this, we also derive a process for the construction of optimal transformations (based on a slightly revised energy). Suppose $\Omega$ is a bounded, smooth, simply connected domain containing the origin. Due to the Riemann Mapping Theorem, there exists a unique (complex)
analytic map $\Psi$ such that $\Psi\left(0\right)=0$, $D\Psi\left(0\right)=aI$
for some $a>0$ and $\Psi$ is a one-to-one mapping from $\overline{\Omega}$ onto
$\overline{B_{1}}$. By the maximum modulus principle $\min\{ |x|\,:\, x \in \overline{\Omega}\} \le 1/a \le \max\{ |x|\,:\,x \in \overline{\Omega}\}$. Set $\omega_{\epsilon}=\Psi^{-1}\left(B_{\epsilon}\right)$, and $\omega_{\frac{1}{2}}=\Psi^{-1}\left(B_{\frac{1}{2}}\right).$ By construction, $0\in\omega_{\epsilon}\subset\omega_{\frac{1}{2}}\subset\Omega$~. 
Provided $\epsilon$ is small enough, $\omega_{\epsilon}$ is approximately
$B_{\frac{\epsilon}{a}}$, in the sense that 
\[
\text{\ensuremath{\forall x\in C_{\epsilon} ~~~\left|\Psi^{-1}\left(x\right)-\frac{x}{a}\right|\le \frac{1}{2}\max_{\overline{B_{1/2}}}\left|D^{2}\text{\ensuremath{\Psi}}^{-1}\right|\epsilon^{2}~.}}
\]

\vskip 10pt
\noindent
Given $\Phi_{\epsilon}\in C^{1}\left(\overline{B_{1}}\setminus B_{\epsilon};\overline{B_{1}}\setminus B_{\frac{1}{2}}\right)$
a (possibly optimal) bijective diffeomorphism with $\left.\Phi_{\epsilon}\right|_{C_{1}}=Id~,\text{ and }\Phi_{\epsilon}\left(C_{\epsilon}\right)=C_{\frac{1}{2}}~$, we define
\begin{equation}
\Psi_{\epsilon}:=\Psi^{-1}\circ\Phi_{\epsilon}\circ\Psi.\label{eq:defnPsieps}
\end{equation}

\begin{figure}
%%%
%
\tikzset{declare function={partx(\x,\y)  =                0.5*sinh(2*\x*cos(deg(\y)));}}
\tikzset{declare function={party(\x,\y)  =                0.5*sin(deg(2*\x*sin(deg(\y))));}}
\tikzset{declare function={denoma(\x,\y) = cosh(\x*cos(deg(\y)))^2 - sin(deg(\x*sin(deg(\y))))^2;}}
\tikzset{declare function={funcy(\x,\y)  =             party(\x,\y)/denoma(\x,\y);}}
\tikzset{declare function={funcx(\x,\y)  =             partx(\x,\y)/denoma(\x,\y);}}
\tikzset{declare function={sinhx(\x,\y)  =                sinh(\x*cos(deg(\y)))*cos(deg(\x*sin(deg(\y));}}
\tikzset{declare function={sinhy(\x,\y)  =                cosh(\x*cos(deg(\y)))*sin(deg(\x*sin(deg(\y));}}
\tikzset{declare function={Engel(\x,\y)  =        2*pi*\x/\y+ pi/\y;}}
\makeatletter
\tikzset{%
  prefix node name/.code={%
    \tikzset{%
      name/.code={\edef\tikz@fig@name{#1 ##1}}
    }%
  }%
}
\makeatother
       \begin{tikzpicture}
       \begin{scope}[shift={(4.5cm,0.1cm)}, prefix node name=BR] 
        \node (xu) at ( 1.5,0) {};
        \node (xb) at (-1.5,0) {};
        \node (yu) at (0,1.5) {};
        \node (yb) at (0,-1.5) {};
        \draw  ( 1,-1pt) -- (1 ,1pt);
        \draw  (-1pt, 1) -- (1pt, 1);
        \draw  (-1pt,-1) -- (1pt,-1);
        \draw  (-1,-1pt) -- (-1,1pt) ;  
        \draw[->]  (-1.5,0) -- (1.5,0) node[above]  {$x$};
        \draw[->]  (0,-1.5) -- (0,1.5) node[left]   {$y$};
        \foreach \x in {0.5,1}
        \draw[domain=0:2*pi,samples=200,  blue, variable=\t]  plot ({sinhx(\x,\t)},{sinhy(\x,\t)});
        \foreach \f in {0,1,...,18}
        \draw[domain=0.5:1,samples=40,  red, variable=\t]  plot ({sinhx(\t,Engel(\f,19))},{sinhy(\t,Engel(\f,19))});
       \end{scope}
      \begin{scope}[shift={(0.1cm,0.1cm)},, prefix node name=BL] 
        \node (xu) at ( 1.5,0) {};
        \node (xb) at (-1.5,0) {};
        \node (yu) at (0,1.5) {};
        \node (yb) at (0,-1.5) {};
        \draw  ( 1,-1pt) -- (1 ,1pt);
        \draw  (-1pt, 1) -- (1pt, 1);
        \draw  (-1pt,-1) -- (1pt,-1);
        \draw  (-1,-1pt) -- (-1,1pt) ;  
        \draw[->]  (-1.5,0) -- (1.5,0) node[above]  {$x$};
        \draw[->]  (0,-1.5) -- (0,1.5) node[left]   {$y$};
        \foreach \x in {0.1,1}
        \draw[domain=0:2*pi,samples=400,  blue, variable=\t]  plot ({sinhx(\x,\t)},{sinhy(\x,\t)});
         \foreach \f in {0,1,...,18}
        \draw[domain=0.1:1,samples=40,  green, variable=\t]  plot ({sinhx(\t,Engel(\f,19))},{sinhy(\t,Engel(\f,19))});
     \end{scope}
        \begin{scope}[shift={(4.5cm,4.5cm)},, prefix node name=TR] 
        \node (xu) at ( 1.5,0) {};
        \node (xb) at (-1.5,0) {};
        \node (yu) at (0,1.5) {};
        \node (yb) at (0,-1.5) {};
        \draw  ( 1,-1pt) -- (1 ,1pt);
        \draw  (-1pt, 1) -- (1pt, 1);
        \draw  (-1pt,-1) -- (1pt,-1);
        \draw  (-1,-1pt) -- (-1,1pt) ;  
        \draw[->]  (-1.5,0) -- (1.5,0) node[above]  {$x$};
        \draw[->]  (0,-1.5) -- (0,1.5) node[left]   {$y$};
       \foreach \x in {0.5,1}
       \draw[domain=0:2*pi,samples=200,  blue, variable=\t]  plot ({\x*cos(deg(\t))},{\x*sin(deg(\t))});
       \foreach \f in {0,1,...,18}
       \draw[domain=0.5:1,samples=40,  red, variable=\t]  plot ({\t*cos(deg(Engel(\f,19)))},{\t*sin(deg(Engel(\f,19)))});
       \end{scope}
       \begin{scope}[shift={(0.1cm,4.5cm)},, prefix node name=TL] 
        \node (xu) at ( 1.5,0) {};
        \node (xb) at (-1.5,0) {};
        \node (yu) at (0,1.5) {};
        \node (yb) at (0,-1.5) {};
        \draw  ( 1,-1pt) -- (1 ,1pt);
        \draw  (-1pt, 1) -- (1pt, 1);
        \draw  (-1pt,-1) -- (1pt,-1);
        \draw  (-1,-1pt) -- (-1,1pt) ;  
        \draw[->]  (-1.5,0) -- (1.5,0) node[above]  {$x$};
        \draw[->]  (0,-1.5) -- (0,1.5) node[left]   {$y$};
        \foreach \x in {0.1,1}
        \draw[domain=0:2*pi,samples=200,  blue, variable=\t] plot ({\x*cos(deg(\t))},{\x*sin(deg(\t))});
        \foreach \f in {0,1,...,18}
        \draw[domain=0.1:1,samples=40,  green, variable=\t]   plot ({\t*cos(deg(Engel(\f,19)))},{\t*sin(deg(Engel(\f,19)))});
       \end{scope}
       \draw[thick,->] (BL yu) -- node[right,midway] {$\Psi$} (TL yb) ;
       \draw[thick,<-] (BR yu) -- node[right,midway] {$\Psi^{-1}$} (TR yb) ;
       \draw[thick,->] (BL xu) .. controls (1.9cm, 0.0cm) and  (2.5cm, 0.0cm) .. (BR xb) node[midway,below] {$\Psi_{\epsilon}$};%   ;
       \draw[thick,->] (TL xu) .. controls (1.9cm, 4.6cm) and  (2.5cm, 4.6cm) .. (TR xb) node[midway,above] {$\Phi_{\epsilon}$};%   ;
       \end{tikzpicture}
      
%%%
%\centering{}\includegraphics{Tikz1}
\caption{\label{fig:Conformal}Cloaking by mapping where $\Omega=\sinh\left(B_{1}\right)$, with $\epsilon=1/10$. }
\end{figure}
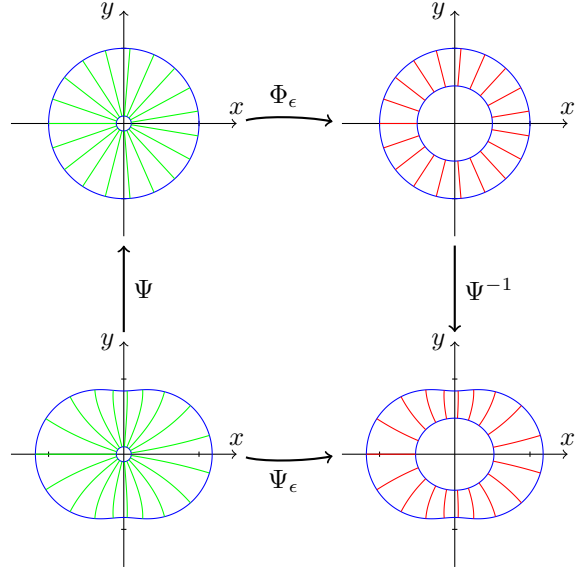
\noindent
Figure~\ref{fig:Conformal} shows some of the ``rays" of the map $\Psi_{\epsilon}$ ($\Phi_\epsilon$ being radial) in the case $\Psi^{-1}=\sinh$, $\Omega=\sinh\left(B_{1}\right)$, $\omega_{\frac{1}{2}}=\sinh\left(B_{\frac{1}{2}}\right)$
and $\omega_{\epsilon}=\sinh\left(B_{\epsilon}\right).$ The green curves on the left are mapped to proper subsets of themselves, shown as red curves on the right. Clearly the transformation $\Psi_\epsilon$ is no longer radial.

\vskip 5pt
\noindent
For any $x\in\partial\Omega$, $\Psi\left(x\right)$ lies on  $C_{1}$, and thus
$\Phi_{\epsilon}\circ\Psi\left(x\right)=\Psi\left(x\right)$. It follows that 
$\Psi_{\epsilon}\left(x\right)=x$, in other words: $\Psi_{\epsilon}=Id$ on
$\partial\Omega$. Similarly, we obtain that  $\Psi_{\epsilon}\left(\partial \omega_{\epsilon}\right)=\partial \omega_{\frac{1}{2}}$. $\left(\Psi_\epsilon\right)_{\star}[I]$ therefore produces an approximate cloak (with same approximate invisibility as that of $(\Phi_\epsilon)_{\star}[I]$). From composition of transformations we obtain 
\[
\left(\Psi_{\epsilon}\right)_{\star}\left[I\right]=\left(\Psi^{-1}\right)_{\star}\left[\left(\Phi_{\epsilon}\right)_{\star}\left[\Psi_{\star}\left[I\right]\right]\right]~.
\]

\begin{lem}\label{lem:conf}
There holds 
\[
\trace\left(\Psi_{\epsilon}\right)_{\star}[I]=\left(\trace\left(\Phi_{\epsilon}\right)_{\star}[I]\right)\circ \Psi~.
\]

\end{lem}

\begin{proof}
Since $\Psi$ is conformal, $D\Psi=\gamma Q$ with
$\gamma$ a positive scalar and $Q$ an orthogonal matrix. We are in 2d, and so this implies
\[
\Psi_{\star}\left[I\right]\left(y\right)=\frac{\left(D\Psi\right)\left(D\Psi\right)^{T}}{|\det D\Psi|}\circ\Psi^{-1}\left(y\right)=I~.
\]
Similarly, 
\begin{align*}
\left(\Psi^{-1}\right)_{\star}\left[A\right]\left(x\right) & =\frac{\left(D\Psi^{-1}\right)A\left(D\Psi^{-1}\right)^{T}}{|\det D\Psi^{-1}|}\circ\Psi\left(x\right)\\
 & =Q^T\left(x\right)A\left(\Psi\left(x\right)\right)Q\left(x\right)~,
\end{align*}
where we have used that $\left(D\Psi^{-1}\right)\left(\Psi(x)\right)= \left(D\Psi\right)^{-1}(x)= \frac{1}{\gamma} Q^T(x)$.
In summary, we conclude that $\left(\Psi_{\epsilon}\right)_{\star}[I]$ is given by the formula
\[
\left(\Psi_{\epsilon}\right)_{\star}\left[I\right]\left(x\right)=Q^T\left(x\right)\left(\Phi_{\epsilon}\right)_{\star}\left[I\right]\left(\Psi\left(x\right)\right)Q\left(x\right)~,
\]
and the statement about the traces follows.
\end{proof}

\noindent
If $\Phi_\epsilon$ is a transformation which minimizes the anisotropy of $(\Phi_\epsilon)_{\star}[I]$, using the measure $I_p$ for some $1\le p<\infty$, then it follows immediately from Lemma~\ref{lem:conf} above that $\Psi_\epsilon$ minimizes anistropy of $(\Psi_\epsilon)_{\star}[I]$, using the slightly modified measure
$$
\tilde I_p\left(\Psi_\epsilon \right) = \int_{\Omega \setminus \omega_{\epsilon}} \left(\trace (\Psi_\epsilon)_*[I]\right)^p(\Psi_\epsilon(x)) \, |\det \Psi(x)|~dx~.
$$ 
A similar statement holds for $p=\infty$. In that case there is no change in the measure of anisotropy.  
\vskip 10pt
{\bf Acknowledgements}
The research of MSV was partially supported by NSF Grant DMS-12-11330. Part of this work was carried out while MSV was visiting the University of Copenhagen and the Danish Technical University. This visit was made possible through support from the Nordea Foundation
and the Otto Mo \hskip -8pt /nsted Foundation. This study also contributes to the IdEx Universit\'e de Paris ANR-18-IDEX-0001.
\bibliographystyle{amsplain}
\bibliography{biblio}
\end{document}